\documentclass[12pt]{amsart}
\usepackage{latexsym}
\usepackage{amssymb, amsmath, mathrsfs,amsfonts}

\newcommand{\A}{\mathcal A}
\newcommand{\B}{\mathcal B}
\newcommand{\M}{\mathcal M}
\newcommand{\C}{\mathcal C}
\newcommand{\T}{\mathcal T}
\newcommand{\N}{\mathbb N}
\newcommand{\Q}{\mathbb Q}
\newcommand{\Z}{\mathbb{Z}}
\newcommand{\SR}{\mathcal{SR}}
\newcommand{\dom}{\text{dom}}

\newcommand{\lqt}{\textquotedblleft}
\newcommand{\rqt}{\textquotedblright~}
\newcommand{\la}{\langle}
\newcommand{\ra}{\rangle}

\newtheorem{theorem}{Theorem}[section]      
\newtheorem{lemma}[theorem]{Lemma}       
\newtheorem{proposition}[theorem]{Proposition}
\newtheorem{corollary}[theorem]{Corollary}

\theoremstyle{definition}
\newtheorem{definition}[theorem]{Definition} 
\newtheorem{example}[theorem]{Example}

\begin{document}

\pagestyle{headings}

\title{Model Theoretic Complexity of Automatic Structures}
\author{Bakhadyr Khoussainov
\and 
Mia Minnes 
}

\maketitle

\begin{abstract}
\noindent We study the complexity of automatic structures via  well-established concepts from both logic and model theory, including ordinal heights (of well-founded relations), Scott ranks  of 
structures, and Cantor-Bendixson ranks (of trees). We prove the following results: 1)~The ordinal height of any automatic well-founded partial order is bounded by $\omega^\omega$; 2)~The ordinal heights of automatic well-founded relations are unbounded below $\omega_{1}^{CK}$, the first non-computable ordinal; 3)~For any computable ordinal $\alpha$, there is an automatic structure of Scott rank at least $\alpha$.  Moreover, there are automatic structures of Scott rank $\omega_1^{CK}, \omega_1^{CK}+1$; 4)~For any computable ordinal $\alpha$, there is an automatic successor tree of Cantor-Bendixson rank $\alpha$.  
\end{abstract}

\section{Introduction}\label{s:Intro}

In recent years, there has been increasing interest in the study of structures 
that can be presented by automata. The underlying idea 
is to apply techniques of automata theory to decision 
problems that arise in logic and applications such as databases and verification. A typical decision problem 
is the model checking problem: for a 
structure $\A$ (e.g.\ a graph), design an algorithm that, given a formula $\phi(\bar{x})$ in a formal system and a tuple $\bar{a}$ from the structure, decides if 
$\phi(\bar{a})$ is true in $\A$. In particular, when the formal system is the first order predicate logic or the monadic second order logic, we would like to know if the 
theory of the structure  is decidable.  Fundamental early results in this direction by B\"uchi (\cite{Buc60}, \cite{Buc62}) and Rabin (\cite{Rab69}) proved the decidability of the monadic second order theories of the successor on the natural numbers and of the binary tree.
There have  been numerous applications and extensions of these results in logic, algebra \cite{EPCHLT92}, verification and model checking \cite{VW84} \cite{Var96}, and databases \cite{Var05}.    Moreover, automatic structures provide a theoretical framework for constraint databases over discrete  domains such as strings and trees \cite{BL02}.  Using simple closure properties and the decidability of the emptiness problem for automata, one can prove that  the first order (and monadic second order) theories of some well-known structures are decidable. Examples of such structures are Presburger arithmetic and some of its extensions, the term algebra, the
real numbers under addition, finitely generated abelian groups, and the atomless Boolean algebra. Direct proofs of  these results, without the use of automata, require non-trivial  technical work.

A structure $\A=(A; R_0, \ldots, R_m)$ is {\bf automatic} if the domain $A$ and all the relations $R_0, \ldots, R_m$ of the structure are recognised by finite automata (precise definitions are in the next section). 
Independently, Hodgson \cite{H82} and later Khoussainov and Nerode \cite{KhN95} proved that for any given automatic structure there is an algorithm that solves the model checking problem for the first order logic. In particular, the first order theory of the structure is decidable.  Blumensath and Gr\"adel proved a logical characterization theorem stating that automatic structures are exactly those definable in the fragment of arithmetic $(\omega; +, |_2, \leq, 0)$, where $+$ and $\leq$ have their usual meanings and $|_2$ is a weak divisibility predicate for which $x|_2 y$ if and only if $x$ is a power of $2$ and divides $y$ (see \cite{BG00}). In addition, for some classes of automatic structures there are characterization theorems that have direct algorithmic implications. For example, in \cite{Del04}, Delhomm\'e proved that automatic well-ordered sets are all strictly less than $\omega^\omega$. Using this characterization, \cite{KhRS03} gives an algorithm which decides the isomorphism problem for automatic well-ordered sets. The algorithm is based on extracting the Cantor normal form for the ordinal isomorphic to the given automatic well-ordered set.  Another characterization theorem of this ilk gives that automatic Boolean algebras are exactly those that are finite products of the Boolean algebra of finite and co-finite subsets of $\omega$ \cite{KhNRS04}. Again, this result can be used to show that the isomorphism problem for automatic Boolean algebras is decidable.  

Another body of work is devoted to the study of resource-bounded complexity of the model checking problem for automatic structures. On the one hand, Gr\"adel and Blumensath (\cite{BG00}) constructed examples of automatic structures whose first order theories are non-elementary. On the other hand, Lohrey in \cite{Loh03} proved that the first order theory of any automatic graph of bounded degree is elementary.  It is noteworthy that when both a first order formula $\phi$ and an automatic structure $\A$ are fixed, determining if a tuple $\bar{a}$ from $\A$ satisfies 
$\phi(\bar{x})$ can be done in linear time. There are also feasible time bounds on deciding the first order theories of automatic structures over the unary alphabet (\cite{Blu99}, \cite{KhLM}). 

Most current results demonstrate that automatic structures are not complex in various concrete senses.
However, in this paper we use well-established concepts from both logic and model theory to prove results in the opposite direction.  We now briefly describe the measures of complexity we use (ordinal heights of well-founded relations, Scott ranks of 
structures, and Cantor-Bendixson ranks of trees) and connect them with the results of this paper.

A relation $R$ is called {\bf well-founded} if there is no infinite sequence $x_1,x_2,x_3, \ldots$ such that $(x_{i+1}, x_{i})\in R$ for $i \in \omega$. In computer science, well-founded relations are of interest due to a natural connection between well-founded sets and terminating programs. 
We say that a program is {\bf terminating} if every computation from an initial state is finite.
This is equivalent to well-foundedness of the collection of states reachable from the initial state, under the reachability relation \cite{BG06}.  The {\bf ordinal height} is a measure of the depth of well-founded relations.  Since all automatic structures are also computable structures, the obvious bound for ordinal heights of automatic well-founded relations is $\omega_1^{CK}$ (the first non-computable ordinal). Sections \ref{s:RanksofOrders} and \ref{s:ranksWF} study the sharpness of this bound.
Theorem \ref{thm:OrderRank} characterizes automatic well-founded partial orders in terms of their (relatively low) ordinal heights, whereas Theorem \ref{thm:HeightRank} shows that $\omega_1^{CK}$ is the sharp bound 
in the general case.

\begin{theorem}\label{thm:OrderRank} For each ordinal $\alpha$, $\alpha$ is the ordinal height of an automatic well-founded partial order if and only if $\alpha< \omega^\omega$. 
\end{theorem}

\begin{theorem}\label{thm:HeightRank}
For each (computable) ordinal $\alpha < \omega_{1}^{CK}$, there is an automatic well-founded relation $\A$ with ordinal height greater than $\alpha$.
\end{theorem}

Section \ref{s:SR} is devoted to building automatic structures with high Scott ranks. The concept of Scott rank comes from a well-known theorem of Scott stating that for every countable structure $\A$ there exists a sentence $\phi$ in $L_{\omega_1,\omega}$-logic which characterizes $\A$ up to isomorphism \cite{Sco65}. The minimal quantifier rank of such a  formula is called the Scott rank of $\A$.  A known upper bound on the Scott rank of computable structures implies that the Scott rank of automatic structures is at most $\omega_1^{CK}+1$.
But, until now,  all the known examples of automatic structures had low Scott ranks. Results in \cite{Loh03}, \cite{Del04}, \cite{KhRS05} 
suggest that the Scott ranks of automatic structures could be bounded by small ordinals. This intuition is falsified in Section \ref{s:SR} with the theorem:

\begin{theorem}\label{thm:ScottRank}
For each computable ordinal $\alpha$
there is an automatic structure of Scott rank at least $\alpha$.
\end{theorem}

In particular, this theorem gives a new proof that the isomorphism problem for automatic structures is $\Sigma_1^1$-complete (another proof may be found in \cite{KhNRS04}).

In the last section, we investigate the Cantor-Bendixson ranks of automatic trees. A
{\bf partial order tree} is  a partially ordered  set $(T, \leq)$ such that  there is a $\leq$-minimal element of $T$, and each subset $\{x \in T : x \leq y\}$ is finite and is linearly ordered
under $\leq$.  A {\bf successor tree} is a pair $(T, S)$ such that the reflexive and transitive closure  $\leq_S$ of $S$ produces a partial order tree 
$(T, \leq_{S})$. The {\bf derivative} of a tree $\T$ is obtained by removing all the nonbranching paths of the tree. One applies the derivative operation to $\T$ successively until a fixed point is reached. The minimal ordinal that is needed to reach the fixed point is called the {\bf Cantor-Bendixson (CB) rank} of the tree. The CB rank plays an important role in logic, algebra, and topology. Informally, the CB rank tells us how far the structure is from algorithmically (or algebraically) simple structures. Again, the obvious bound on $CB$ ranks of automatic successor trees is $\omega_1^{CK}$.  
In \cite{KhRS03}, it is proved that the CB rank of any automatic partial order tree is finite and can be computed from the automaton for the $\leq$ relation on the tree. It has been an open question whether
the CB ranks of automatic successor trees can also be bounded by small ordinals. We answer this question in the following theorem.

\begin{theorem}\label{thm:RecTrees}
For $\alpha < \omega_1^{CK}$ there is an automatic successor tree of CB rank $\alpha$.
\end{theorem}

The main tool we use to prove results about high ranks is the configuration spaces of Turing machines, considered as automatic graphs.
It is important to note that graphs which arise as configuration spaces have very low model-theoretic complexity: their Scott ranks are at most $3$, and if they are well-founded then their ordinal heights are at most $\omega$ (see Propositions \ref{ppn:ConfigWF} and \ref{ppn:ConfigScott}). Hence, the configuration spaces serve merely as building blocks in the construction of automatic structures with high complexity, rather than contributing materially to the high complexity themselves.

\section*{Acknowledgement}
We thank Moshe Vardi who posed the question about ranks of automatic well-founded relations. We also thank Anil Nerode and Frank Stephan with whom 
we discussed Scott and Cantor-Bendixson ranks of automatic structures.

\section{Preliminaries}\label{s:Prelim}

A (relational) {\bf vocabulary} is a finite sequence $(P_1^{m_1}, \ldots, P_t^{m_t}, c_1, \ldots, c_s)$, where each $P_j^{m_j}$ is a predicate symbol of arity $m_j>0$, and each $c_k$ is a constant symbol.  A {\bf structure} with this vocabulary is a tuple $\A=(A;P_1^{\A}, \ldots, P_t^{\A}, c_1^{\A}, \ldots, c_s^{\A})$, where $P_j^{\A}$ and $c_k^{\A}$ are interpretations of the symbols of the vocabulary.  When convenient, we may omit the superscripts $\A$. We only consider infinite structures, that is, those whose universe is an infinite set.

To establish notation, we briefly recall some definitions associated with finite automata.  A {\bf finite automaton} $\M$ over an alphabet $\Sigma$ is a tuple
$(S,\iota,\Delta,F)$, where $S$ is a finite set of {\bf states}, $\iota \in S$
is the {\bf initial state}, $\Delta \subset S \times \Sigma \times S$ is the
{\bf transition table}, and $F \subset S$ is the set of {\bf final states}.
A {\bf computation} of $\A$ on a word $\sigma_1 \sigma_2 \dots \sigma_n$
($\sigma_i \in \Sigma$) is a sequence of states, say $q_0,q_1,\dots,q_n$, such
that $q_0 = \iota$ and $(q_i,\sigma_{i+1},q_{i+1}) \in \Delta$ for all $i \in
\{0,\ldots,n-1\}$.  If $q_n \in F$, then the computation is {\bf successful}
and we say that the automaton $\M$ {\bf accepts} the word $\sigma_1 \sigma_2 \dots \sigma_n$. The {\bf language}
accepted by the automaton $\M$ is the set of all words accepted by $\M$. In
general, $D \subset \Sigma^{\star}$ is {\bf finite automaton recognisable},
or {\bf regular}, if $D$ is the language accepted by some finite automaton~$\M$.

To define  automaton recognisable  relations, we use $n$-variable (or $n$-tape) automata.
An {\bf $n$--tape automaton} can be thought of as a one-way 
Turing machine with $n$ input tapes \cite{Eil69}.  Each tape is regarded as 
semi-infinite, having written on it a word over the alphabet $\Sigma$ followed 
by an infinite succession of blanks (denoted by $\diamond$ symbols).  The automaton 
starts in the initial state, reads simultaneously the first symbol of each tape, 
changes state, reads simultaneously the second symbol of each tape, 
changes state, etc., until it reads a blank on each tape. The automaton then 
stops and accepts the $n$--tuple of words if it is in a final state. The set of all 
$n$--tuples accepted by the automaton is the relation recognised by the automaton.  
Formally,  an $n$--tape  automaton on $\Sigma$ is a finite automaton over the alphabet $(\Sigma_{\diamond})^n$, where $\Sigma_{\diamond}=\Sigma \cup \{\diamond\}$ and
$\diamond \not \in \Sigma$.
The {\bf convolution of a tuple} $(w_1,\cdots,w_n) \in
\Sigma^{\star n}$ is the string $ c(w_1,\cdots,w_n)$ of length $\max_i|w_i|$
over the alphabet $(\Sigma_{\diamond})^n$ which is defined as follows.
Its $k$'th symbol is $(\sigma_1,\ldots,\sigma_n)$ where $\sigma_i$ is the
$k$'th symbol of $w_i$ if $k \leq |w_i|$ and $\diamond$ otherwise.
The {\bf convolution of a relation} $R \subset \Sigma^{\star n}$ is the language
$c(R) \subset (\Sigma_{\diamond})^{n\star}$ formed as the set of convolutions
of all the tuples in $R$. 
 An $n$--ary relation $R \subset \Sigma^{{\star}n}$ is {\bf finite automaton recognisable},
  or {\bf regular}, if its convolution $c(R)$ is recognisable by an $n$--tape automaton.

\begin{definition} \label{dfn:automatic} A structure $\A=(A; R_0, R_1, \ldots, R_m)$ is {\bf automatic} over $\Sigma$ if its domain $A$ and all relations 
$R_0$, $R_1$, $\ldots$, $R_m$  are regular over $\Sigma$.  If $\B$ is isomorphic to an automatic structure $\A$
then we call $\A$ an {\bf automatic presentation} of $\B$ and say that
$\B$ is {\bf automatically presentable}.
\end{definition}

The configuration graph of any Turing machine
is an example of an automatic structure. The graph is defined by letting the 
configurations of the Turing machine be the vertices, 
and putting an edge from configuration $c_1$ to configuration $c_2$ if the machine can make an instantaneous move from $c_1$ to $c_2$. Examples of automatically
presentable structures are $(\N, +)$, $(\N, \leq)$, $(\N, S)$,
$(\Z,+)$, the order on the rationals $(Q, \leq)$, and the Boolean algebra
of finite and co-finite subsets of $\N$. In the following, we abuse terminology and identify the notions of \lqt automatic\rqt and \lqt automatically presentable\rqt.
Many examples of automatic structures can be formed using the $\omega$-fold disjoint union of a structure $\A$ (the disjoint union of $\omega$ many  
copies of $\A$). 

\begin{lemma}\cite{Rub04}\label{lm:omega-fold} If $\A$ is automatic then its $\omega$-fold disjoint union is isomorphic to an automatic structure. 
\end{lemma}
\begin{proof}
Suppose that $\A = (A; R_{1}, R_2, \ldots)$ is automatic.  Define $\A' = (A \times 1^{\star}; R'_{1}, R_2', \ldots)$ by
\[
\la (x,i), (y, j) \ra \in R'_{m} \qquad \iff \qquad i = j~ \& ~\la x, y \ra \in R_{m},  \   \  m=1,2,\ldots.
\]
It is clear that $\A'$ is automatic and is isomorphic to the $\omega$-fold disjoint union of $\A$. 
\end{proof}

The class of automatic structures is a proper subclass of the computable structures.  
We therefore mention some crucial definitions and facts about computable structures.
Good references for the theory of computable structures include \cite{Hariz98}, \cite{KhSh99}.  

\begin{definition}
A {\bf computable structure} is $\A = (A; R_{1}, \ldots, R_m)$ whose domain and relations are all computable.  
\end{definition}

 The domains of computable structures can always be identified with the set $\omega$ of natural numbers. Under this assumption, we introduce  new constant symbols 
 $c_n$ for each $n\in \omega$ and interpret $c_n$ as $n$. We expand the vocabulary of each structure to include these new constants $c_{n}$.   In this context, $\A$  is computable if and only if 
  the {\bf atomic diagram} of $\A$ (the set of G\"odel numbers of all quantifier-free sentences in the extended vocabulary that are true in $\A$) is a computable set.
  If $\A$ is computable and $\B$ is isomorphic to $\A$ then we say that $\A$ is a {\bf computable presentation}
of $\B$. Note that if $\B$ has a computable presentation then $\B$ has  $\omega$ many computable presentations.  In this paper, we will be coding computable structures into automatic ones.

The ranks that we use to measure the complexity of automatic structures take values in the ordinals.  In particular, we will see that only a subset of the countable ordinals will play an important role.  An ordinal is called {\bf computable} if it is the order-type of a computable well-ordering of the natural numbers.  The least ordinal which is not computable is denoted $\omega_1^{CK}$ (after Church and Kleene). 

\section{Ranks of automatic well-founded partial orders} \label{s:RanksofOrders}

In this section we consider structures $\A = (A; R)$ with a single binary relation.  
An element $x$ is said to be {\bf $R$-minimal for a set $X$} if for each 
$y \in X$, $(y,x) \notin R$.  The relation $R$ is said to be {\bf well-founded} 
if every non-empty  subset of $A$ has an $R$-minimal element.  
This is equivalent to saying that $(A; R)$ has no infinite chains 
$x_1, x_2, x_3, \ldots$ where $(x_{i+1}, x_i) \in R$ for all $i$.

 A {\bf ranking function} for $\A$ is an ordinal-valued function $f$ such that $f(y)< f(x)$ whenever $(y,x)\in R$. If $f$ is a ranking function on $\A$, let $ord(f)= \sup\{ f(x) : x \in A \}$.  The structure $\A$ is well-founded if and only if  $\A$ admits a ranking function.  The {\bf ordinal height} of $\A$, denoted $r(\A)$, is the least ordinal $\alpha$ which is $ord(g)$ for some ranking function $g$ on $\A$.  An equivalent definition for the rank of $\A$ is the following.  We define the function $r_{\A}$ by induction: for the $R$-minimal elements $x$,
set $r_{\A}(x)=0$; for $z$ not $R$-minimal, put $r_{\A}(z)=\sup\{ r(y)+1 : (y,z) \in R\}$. Then $r_{\A}$ is a ranking function admitted by $\A$ and $r(\A) = \sup\{ r_{\A}(x) : x \in A\}$.
For $B \subseteq A$, we write $r(B)$ for the ordinal height of the structure obtained by restricting the relation $R$ to the subset $B$. 

\begin{lemma}\label{lm:compRank}
If $\alpha<\omega_1^{CK}$, there is a computable well-founded relation of 
ordinal height $\alpha$.
\end{lemma}
\begin{proof}
This lemma is trivial: the ordinal height of an ordinal $\alpha$ is $\alpha$ itself.  Since all computable ordinals are computable and well-founded relations, we are done. 
\end{proof}
  
The next lemma follows easily from the well-foundedness of ordinals and of $R$.  The proof is left to the reader.

\begin{lemma}\label{lm:witnessRank}
For a structure $\A = (A; R)$ where $R$ is well-founded, if $r(\A) = \alpha$ and $\beta < \alpha$ then there is an $x \in A$ such that $r_{\A}(x) = \beta$. 
\end{lemma}

For the remainder of this section, we assume further that $R$ is a partial order. For
convenience, we write $\leq$ instead of $R$. Thus,  we consider automatic well-founded partial orders $\A=(A,\leq)$.  We will use the notion of {\bf natural sum of ordinals}. The natural sum of ordinals $\alpha, \beta$ (denoted $\alpha +' \beta$) is defined recursively: $\alpha +' 0 = \alpha$, $0 +' \beta = \beta$, and $\alpha +' \beta$ is the least ordinal strictly greater than $\gamma +' \beta$ for all $\gamma < \alpha$ and strictly greater than $\alpha +' \gamma$ for all $\gamma < \beta$.

\begin{lemma}
Let $A_1$ and $A_2$ be disjoint subsets of $A$ such that $A=A_1\cup A_2$.  
Consider the partially ordered sets $\A_1=(A_1,\leq_1)$ and $\A_2=(A_2,\leq_2)$ obtained by restricting $\leq$ to $A_1$ and $A_2$ respectively. Then, $r(\A)\leq \alpha_1 +' \alpha_2$, where $\alpha_i=r(\A_i)$. 
\end{lemma}
\begin{proof}
We will show that there is a ranking function on $A$ whose range is contained in the ordinal
$\alpha_1 +' \alpha_2$.   
For each $x\in A$
consider the partially ordered sets $\A_{1,x}$ and $\A_{2,x}$ obtained by restricting
$\leq$ to $\{z\in A_1 \mid z < x\}$ and $\{z\in A_2 \mid z < x\}$, respectively. 
Define $f(x)=r(\A_{1,x}) +' r(\A_{2,x})$.
We claim that $f$ is a ranking function. Indeed, assume that $x<y$. Then, since $\leq$ is transitive, it must be the case that $\A_{1,x} \subseteq A_{1,y}$ and $\A_{2,x} \subseteq A_{2,y}$. Therefore, $r(\A_{1,x})\leq r(\A_{1,y})$ and $r(\A_{2,x})\leq r(\A_{2,y})$. At least one of these inequalities must be strict. To see this, assume that $x\in A_1$ (the case $x\in A_2$ is similar). Then since $x\in A_{1,y}$, 
it is the case that $r(\A_{1,x})+1 \leq  r(\A_{1,y})$ by the definition of ranks. Therefore, we have that $f(x) < f(y)$. 
Moreover, the image of $f(x)$ is contained in $\alpha_1 +' \alpha_2$.
 \end{proof}

\begin{corollary} \label{cr:PartitionRank}
If $r(\A)=\omega^n$ and $A=A_1\cup A_2$, where $A_1\cap A_2=\emptyset$,
then either $r(\A_1)=\omega^n$ or $r(\A_2)=\omega^n$. 
\end{corollary}

Khoussainov and Nerode \cite{KhN95} show that, for each $n$, there is an automatic presentation of the ordinal $\omega^n$.  It is clear that such a presentation has ordinal height $\omega^n$. The next theorem proves that $\omega^\omega$ is the sharp bound on ranks of all automatic well-founded partial orders.  Once Corollary \ref{cr:PartitionRank} has been established, the proof of Theorem \ref{thm:OrderRank} follows Delhomm\'e \cite{Del04} and Rubin \cite{Rub04}.\\

\noindent{\bf Theorem \ref{thm:OrderRank}.~}{\em
For each ordinal $\alpha$, $\alpha$ is the ordinal height of an automatic well-founded partial order if and only if $\alpha< \omega^\omega$.}

\begin{proof}
One direction of the proof is clear. For the other, assume for a contradiction that there is an automatic well-founded partial order $\A = (A, \leq)$ with $r(\A) = \alpha\geq\omega^\omega$.  Let $(S_{A}, \iota_{A}, \Delta_{A}, F_{A})$ and $(S_{\leq}, \iota_{\leq}, \Delta_{\leq}, F_{\leq})$ be finite automata over $\Sigma$ recognizing $A$ and $\leq$ (respectively).  
 By Lemma \ref{lm:witnessRank}, for each $n>0$ there is
$u_n\in A$ such that $r_{\A}(u_n)=\omega^n$.   For each $u \in A$ we define the set
\[
u \downarrow = \{ x \in A : x < u \}.
\]
Note that if $r_{\A}(u)$ is a limit ordinal then $r_{\A}(u) = r(u\downarrow)$. We define a finite partition of $u \downarrow$ in order to apply Corollary \ref{cr:PartitionRank}.  To do so, 
for $u, v \in \Sigma^{\star}$, define
$X_{v}^{u} = \{ vw \in A : w \in \Sigma^{\star} \ \& \ vw < u \}$.  
Each set of the form $u \downarrow$ can then be partitioned based on the prefixes of words
as follows:
\[
u \downarrow = \{ x \in A : |x| < |u | \ \& \ x < u \} \cup \bigcup_{v \in \Sigma^{\star} : |v| = |u|} X_{v}^{u}.
\]
(All the unions above are finite and disjoint.)  Hence, applying Corollary \ref{cr:PartitionRank}, for each $u_n$ there exists  a $v_n$ such that $|u_n|=|v_n|$ and $r(X_{v_n}^{u_n})=r(u_n \downarrow)=\omega^n$.

On the other hand, we use the automata to define the following equivalence relation on pairs of words of equal lengths:
\begin{align*}
(u,v) \sim (u', v') \ \iff \ &\Delta_{A}(\iota_{A}, v) = \Delta_{A}(\iota_{A}, v') \ \& \\ &\Delta_{\leq}(\iota_{\leq}, \binom{v}{u}) = \Delta_{\leq}(\iota_{\leq}, \binom{v'}{u'})
\end{align*}
There are at most  $|S_{A}|\times |S_{\leq}|$ equivalence classes. Thus, the infinite sequence $(u_1, v_1)$, $(u_2, v_2)$, $\ldots$ contains $m$, $n$ such that $m \neq n$ and $(u_{m}, v_{m}) \sim (u_{n}, v_{n})$.  

\begin{lemma}\label{lm:IsoXvu}
For any $u,v,u',v' \in \Sigma^{\star}$, if $(u,v) \sim (u', v')$ then $r(X_{v}^{u}) = r(X_{v'}^{u'})$.
\end{lemma}

To prove the lemma, consider $g: X_{v}^{u} \to X_{v'}^{u'}$ defined as $g(vw) = v'w$.  From the equivalence relation, we see that $g$ is well-defined, bijective, and order preserving.  Hence $X_v^u \cong X_{v'}^{u'}$ (as partial orders).  Therefore, $r(X_{v}^{u}) = r(X_{v'}^{u'})$.

By Lemma \ref{lm:IsoXvu}, $\omega^{m} = r(X_{v_{m}}^{u_{m}}) = r(X_{v_{n}}^{u_{n}}) = \omega^{n}$, a contradiction with the assumption that $m \neq n$.  Therefore, there is no automatic well-founded partial order of ordinal height greater than or equal to $\omega^{\omega}$. 
\end{proof}

\section{Ranks of automatic well-founded relations}\label{s:ranksWF}

\subsection{Configuration spaces of Turing machines}\label{s:Config}
In the forthcoming constructions, we embed computable structures into 
automatic ones via configuration spaces of  Turing machines.
This subsection provides terminology and background for these constructions. 
Let $\M$ be an $n$-tape deterministic Turing machine.  
The {\bf configuration space}  of $\M$, denoted by $Conf(\M)$, is a directed graph whose nodes are configurations of $\M$. The nodes are $n$-tuples, each of whose coordinates 
represents the contents of a tape.  Each tape is encoded as $(w ~q ~ w')$, 
where $w, w' \in \Sigma^{\star}$ are the symbols on the tape before and 
after the location of the read/write head, and $q$ is one of the  states 
of $\M$.  The edges of the graph are all the pairs of the form $(c_1,c_2)$ such that 
there is an instruction of $\M$ that  transforms  
$c_{1}$ to $c_{2}$.  The configuration space is an automatic graph.  The out-degree of every vertex in $Conf(\M)$ is $1$; the in-degree need not be $1$. 

\begin{definition}
A deterministic Turing machine $\M$ is {\bf reversible} if $Conf(\M)$ consists only of finite chains and chains of type $\omega$.
\end{definition}

\begin{lemma} \cite{Ben73} \label{lm:reverse}
For any deterministic $1$-tape Turing machine there is a reversible $3$-tape Turing machine which accepts the same language.
\end{lemma}

\begin{proof}(Sketch)
Given a deterministic Turing machine, define a $3$-tape Turing machine
with a modified set of instructions. 
The modified instructions have the property that neither the domains nor the ranges overlap.  The first tape performs the computation exactly as the original machine would have done.  As the new machine executes each instruction, it stores the index of the instruction on the second tape, forming a history.  Once the machine enters a state which would have been halting for the original machine, the output of the computation is copied onto the third tape.  Then, the machine runs the computation backwards  and erases the history tape.  The halting configuration contains the input on the first tape, blanks on the second tape, and the output on the third tape.
\end{proof}

We establish the following notation for a $3$-tape reversible Turing machine $\M$ given by the construction in this lemma.  A {\bf valid initial configuration} of $\M$ is of the form $(\lambda~ \iota ~ x , \lambda, \lambda )$, where $x$ in the domain,  $\lambda$ is the empty string,  and $\iota$ is the initial state of $\M$. From the  proof of Lemma \ref{lm:reverse}, observe that a {\bf final (halting)} configuration is of the form $(x, \lambda, \lambda ~q_{f} ~y)$, 
with $q_{f}$ a halting state of $\M$.   Also, because of the reversibility assumption, all the chains in $Conf(\M)$ 
are either finite or $\omega$-chains (the order type of the natural numbers).  In particular, this means that $Conf(\M)$ is well-founded.  We call an element of in-degree $0$ a {\bf base} (of a chain).  The set of valid initial or final configurations is regular. We classify the components (chains) of $Conf(\M)$ as follows:
\begin{itemize}
\item {\bf Terminating computation chains}: finite chains whose base is a valid initial configuration; that is, one of the form $(\lambda ~\iota~ x, \lambda, \lambda )$, for $x \in \Sigma^{\star}$.
\item {\bf Non-terminating computation chains}: infinite chains whose base is a valid initial configuration. 
\item {\bf Unproductive chains}: chains whose base is not a valid initial configuration.
\end{itemize} 

Configuration spaces of reversible Turing machines are locally finite graphs (graphs of finite degree) and well-founded.  Hence, the following proposition guarantees that their ordinal heights are small.

\begin{proposition} \label{ppn:ConfigWF} 
If $G = (A,E)$ is a locally finite graph then $E$ is well-founded and the ordinal height of $E$ is not above $\omega$, or $E$ has an infinite chain.
\end{proposition}

\begin{proof}
Suppose $G$ is a locally finite graph and $E$ is well-founded.  For a contradiction, suppose $r(G) > \omega$.  Then there is $v \in A$ with $r(v) = \omega$. By definition, $r(v) = \sup\{ r(u) : u E v \}$.  But, this implies that there are infinitely many elements $E$-below $v$, a contradiction with local finiteness of $G$.
\end{proof}

\subsection{Automatic well-founded relations of high rank}

We are now ready to prove that $\omega_1^{CK}$ is the sharp bound for ordinal heights of automatic well-founded relations.

\vspace{5pt}

\noindent{\bf Theorem \ref{thm:HeightRank}.~}{\em
For each computable ordinal $\alpha < \omega_{1}^{CK}$, there is an automatic well-founded relation $\A$ with ordinal height greater than $\alpha$}


\begin{proof} 
The proof of the theorem uses properties of Turing machines and their configuration spaces.  We take a computable well-founded relation whose ordinal height is $\alpha$, and ``embed" it into an automatic well-founded relation with similar ordinal height.

By Lemma \ref{lm:compRank}, let $\C=(C, L_\alpha)$ be a computable well-founded relation of ordinal height $\alpha$.  
We assume without loss of generality that $C = \Sigma^{\star}$ for some finite alphabet $\Sigma$. Let $\M$ be the Turing machine computing the relation $L_{\alpha}$.  On each pair $(x,y)$ from the domain, $\M$ halts and outputs \lqt yes\rqt or \lqt no\rqt. By Lemma \ref{lm:reverse}, we can assume that $\M$ is reversible.  Recall that $Conf(\M) = (D, E)$ is an automatic graph. 
We define the domain of our automatic structure to be $A = \Sigma^{\star} \cup D$.  The binary relation of the automatic structure is:
\begin{align*}
R = E ~\cup~ &\{ (x, (\lambda ~ \iota ~ (x, y), \lambda, \lambda) ) : x,y \in \Sigma^{\star}\}  ~\cup \\
&\{ (( (x,y), \lambda, \lambda~q_{f}~\text{\lqt yes\rqt}), y) : x,y \in \Sigma^{\star}\}.
\end{align*}
Intuitively, the structure $(A; R)$ is a stretched out version of $(C, L_\alpha)$ with infinitely many finite pieces extending from elements of $C$, and with disjoint pieces which are either finite chains or chains of type $\omega$.  The structure $(A; R)$ is automatic because its domain is a regular set of words
and the relation $R$ is recognisable by a $2$-tape automaton.  We should verify, however, that $R$ is well-founded.  Let $Y \subset A$.  If $Y \cap C \neq \emptyset$ then since $(C, L_{\alpha})$ is well-founded, there is $x \in Y \cap C$ which is $L_{\alpha}$-minimal. The only possible elements $u$ in $Y$ for which $(u,x) \in R$ are those which lie on computation chains connecting some $z \in C$ with $x$.  Since each such computation chain is finite, there is an $R$-minimal $u$ below $x$ on each chain.  Any such $u$ is $R$-minimal for $Y$.  On the other hand, if $Y \cap C = \emptyset$, then $Y$ consists of disjoint finite chains and chains of type $\omega$. Any such chain has a minimal element, and any of these elements are $R$-minimal for $Y$. 
Therefore, $(A; R)$ is an automatic well-founded structure.

We now consider the ordinal height of $(A; R)$.  
For each element $x \in C$, an easy induction on $r_{C}(x)$, shows that 
$r_{\C} (x) \leq r_{\A}(x) \leq \omega+r_{\C} (x)$. 
We denote by $\ell(a,b)$ the (finite) length of the computation chain of $\M$ with input $(a,b)$.  For any element $a_{x,y}$ in the computation chain which represents the computation of $\M$ determining whether $(x,y) \in R$, we have \ 
$r_{\A}(x) \leq r_{\A}(a_{x,y}) \leq r_{\A}(x) + \ell(x,y)$. \ 
For any element $u$ in an unproductive chain of the configuration space,  $0\leq r_{\A}(u)<\omega$.  Therefore, since $C \subset A$,  \ 
$r(\C) \leq r(\A) \leq \omega + r(C)$.
\end{proof}

\section{Automatic Structures and Scott Rank}\label{s:SR}
  
 The Scott rank of a structure is introduced in the proof of Scott's Isomorphism Theorem \cite{Sco65}. Since then, variants of the Scott rank have been used in the computable model theory literature. Here we follow the definition of Scott rank from \cite{CGKn05}.

\begin{definition}
For structure $\A$ and tuples $\bar{a}, \bar{b} \in A^{n}$ (of equal length), define
\begin{itemize}
\item $\bar{a} \equiv^{0} \bar{b}$ if $\bar{a}, \bar{b}$ satisfy the same quantifier-free formulas in the language of $\A$; 
\item For $\alpha > 0$, $\bar{a} \equiv^{\alpha} \bar{b}$ if for all $\beta < \alpha$, for each $\bar{c}$ (of arbitrary length) there is $\bar{d}$ such that 
$\bar{a}, \bar{c} \equiv^{\beta} \bar{b}, \bar{d}$; and for each $\bar{d}$ (of arbitrary length) there is $\bar{c}$  such that 
$\bar{a}, \bar{c} \equiv^{\beta} \bar{b}, \bar{d}$.
\end{itemize}
Then, the {\bf Scott rank} of the tuple $\bar{a}$, denoted by $\SR(\bar{a})$, is  the least 
$\beta$ such that for all $\bar{b} \in A^{n}$, $\bar{a}\equiv^{\beta} \bar{b}$ implies that $(\A, \bar{a}) \cong (\A, \bar{b})$.  Finally, the Scott rank of $\A$, denoted by $\SR(\A)$, is the
least $\alpha$ greater than the Scott ranks of all tuples of $\A$.
\end{definition}

\begin{example} $\SR(\Q, \leq) = 1$, $\SR(\omega, \leq) = 2$, and $\SR( n \cdot \omega, \leq) = n+1$.
\end{example}

Configuration spaces of reversible Turing machines are locally finite graphs.  By the proposition below, they all have low Scott Rank.
\begin{proposition}\label{ppn:ConfigScott}
If $G = (V,E)$ is a locally finite graph, $SR(G) \leq 3$.
\end{proposition}
\begin{proof}
The neighbourhood of diameter $n$ of a subset $U$, denoted $B_{n}(U)$, is defined as follows: $B_0(U) = U$ and $B_n(U)$ is the set of $v \in V$ which can be reached from $U$ by $n$ or fewer edges.  The proof of the proposition relies on two lemmas.

\begin{lemma}\label{lm:ConfigScott_1}
Let $\bar{a}, \bar{b} \in V$ be such that $\bar{a} \equiv^2 \bar{b}$.  Then for all $n$, there is a bijection of the $n$-neighbourhoods around $\bar{a}, \bar{b}$ which sends $\bar{a}$ to $\bar{b}$ and which respects $E$.
\end{lemma}
\begin{proof}
For a given $n$, let $\bar{c} = B_{n}(\bar{a})\setminus \bar{a}$. Note that $\bar{c}$ is a finite tuple because of the local finiteness condition.  Since $\bar{a} \equiv^2 \bar{b}$, there is $\bar{d}$ such that $\bar{a} \bar{c} \equiv^1 \bar{b} \bar{d}$.  If $B_{n}(\bar{b}) = \bar{b} \bar{d}$, we are done.  Two set inclusions are needed.  First, we show that $d_{i} \in B_{n}(\bar{b})$.  By definition, we have that $c_{i} \in B_{n}(\bar{a})$, and let $a_{j}, u_{1}, \ldots, u_{n-1}$ witness this.  Then since $\bar{a} \bar{c} \equiv^1 \bar{b} \bar{d}$, there are $v_{1}, \ldots, v_{n-1}$ such that $\bar{a}\bar{c} \bar{u}\equiv^0 \bar{b}\bar{d}\bar{v}$.  In particular, we have that if $c_{i} E u_{i} E \cdots E u_{n-1} E a_{j}$, then also $d_{i} E v_{i} E \cdots E v_{n-1} E b_{j}$ (and likewise if the $E$ relation is in the other direction).  Hence, $d_{i} \in B_{n}(\bar{b})$.  Conversely, suppose $v \in B_{n}(\bar{b}) \setminus \bar{d}$.  Let $v_{1}, \ldots, v_{n}$ be witnesses and this will let us find a new element of $B_{n}(\bar{a})$ which is not in $\bar{c}$, a contradiction.
\end{proof}

\begin{lemma}\label{lm:ConfigScott_2}
Let $G=(V,E)$ be a graph.  Suppose $\bar{a}, \bar{b} \in V$ are such that for all $n$, $(B_{n}(\bar{a}), E, \bar{a}) \cong (B_{n}(\bar{b}), E, \bar{b})$.  Then there is an isomorphism between the component of $G$ containing $\bar{a}$ and that containing $\bar{b}$ which sends $\bar{a}$ to $\bar{b}$.
\end{lemma}
\begin{proof}
We consider a tree of partial isomorphisms of $G$.  The nodes of the tree are bijections from $B_{n}(\bar{a})$ to $B_{n}(\bar{b})$ which respect the relation $E$ and map $\bar{a}$ to $\bar{b}$.  Node $f$ is the child of node $g$ in the tree if $\text{dom}(f) = B_{n}(\bar{a})$, $\text{dom}(g)  = B_{n+1}(\bar{a})$ and $f \supset g$.  Note that the root of this tree is the map which sends $\bar{a}$ to $\bar{b}$.  Moreover, the tree is finitely branching and is infinite by Lemma \ref{lm:ConfigScott_1}.  Therefore, K\"onig's Lemma gives an infinite path through this tree.  The union of all partial isomorphisms along this path is the required isomorphism.
\end{proof}

To prove the proposition, we note that for any $\bar{a}, \bar{b}$ in $V$ such that $\bar{a} \equiv^2 \bar{b}$, Lemmas \ref{lm:ConfigScott_1} and \ref{lm:ConfigScott_2} yield an isomorphism from the component of $\bar{a}$ to the component of $\bar{b}$ that maps $\bar{a}$ to $\bar{b}$.  Hence, if $\bar{a} \equiv^2 \bar{b}$, there is an automorphism of $G$ that maps $\bar{a}$ to $\bar{b}$.  Therefore, for each $\bar{a} \in V$, $SR(\bar{a}) \leq 2$, so $SR(G) \leq 3$. 
\end{proof}

Let $\C=(C; R_{1}, \ldots, R_{m})$ be a computable structure.  Recall that since $C$ is a computable set, we may assume it is $\Sigma^\star$ for some finite alphabet $\Sigma$.  We construct an 
automatic structure $\A$ whose Scott rank is (close to) the Scott rank of $\C$.  Since the domain of $\C$ is computable, we assume that $C=\Sigma^{\star}$ for some finite 
$\Sigma$.  The construction of $\A$ involves connecting the configuration spaces of Turing machines computing relations $R_{1}, \ldots, R_{m}$. Note that Proposition \ref{ppn:ConfigScott} suggests that the high Scott rank of the resulting automatic structure is the main part of the construction because it is not provided by the configuration spaces themselves.  The construction in some sense expands $\C$ into an automatic structure.  We comment that expansions do not  necessarily preserve the Scott rank.  For example, any computable structure, $\C$, has an expansion with Scott rank $2$.  The expansion is obtained by adding the successor relation into the signature.

 We detail the construction for  $R_{i}$.  Let $\M_{i}$ be a Turing machine for $R_{i}$. 
By a simple modification of the machine we assume that $\M_{i}$ halts if and only if its output is \lqt yes\rqt.  By Lemma \ref{lm:reverse}, we can also assume that $\M_{i}$ is reversible.  We now modify the configuration space $Conf(\M_{i})$ so as to respect the isomorphism type of $\C$. 
This will ensure that the construction (almost) preserves the Scott rank of $\C$. We use the terminology from Subsection \ref{s:Config}.

{\bf Smoothing out unproductive parts}. The length and number of unproductive chains is determined by the machine $\M_{i}$ and hence may differ even for Turing machines computing the same set.  In this stage, we standardize the format of this unproductive part of the configuration space.  We wish to add enough redundant information in the unproductive section of the structure so that if two given computable structures are isomorphic, the unproductive parts of the automatic representations will also be isomorphic .
We add $\omega$-many chains of length $n$ (for each $n$) and $\omega$-many copies of $\omega$.  This ensures that the (smoothed) unproductive section of the configuration space of any Turing machine will be isomorphic and preserves automaticity.   We comment that adding this redundancy preserves automaticity since the operation is a disjoint union of automatic structures.

{\bf Smoothing out lengths of computation chains}. We turn our attention to the chains  which have valid initial configurations at their base.  The length of each finite chain denotes the length of computation required to return a \lqt yes\rqt answer.  We will smooth out these chains by adding \lqt fans\rqt to each base.  For this, we connect to each base of a computation chain a structure which consists of $\omega$ many chains of each finite length.  To do so we follow Rubin \cite{Rub04}:  consider the structure whose domain is $0^{\star} 0 1^{\star}$ and whose relation is given by $x E y$ if and only if $|x| = |y|$ and $y$ is the least lexicographic successor of $x$. This structure has a finite chain of every finite length.  As in Lemma \ref{lm:omega-fold}, we take the $\omega$-fold disjoint union of the structure and identify the bases of all the finite chains.  We get a \lqt fan\rqt with infinitely many chains of each finite size whose base can be identified with a valid initial computation state.  Also, the fan has an infinite component if and only if $R_{i}$ does not hold of the input tuple corresponding to the base. The result is an automatic graph, $Smooth(R_{i}) = ( D_{i}, E_{i})$, which extends $Conf(\M_{i})$.

{\bf Connecting domain symbols to the computations of the relation}. 
We apply the construction above to each  $R_{i}$ in the signature of $\C$. 
Taking the union of the resulting automatic graphs and adding vertices for the domain, we have the structure $(\Sigma^{\star} \cup \cup_i D_{i}, E_{1}, \ldots, E_{n})$ (where we assume that the $D_i$ are disjoint).
We assume without loss of generality that each $\M_{i}$ has a different initial state, and denote it by $\iota_{i}$.   We add $n$ predicates $F_i$ to the signature of the automatic structure connecting the elements of the domain of $\C$ with the computations of the relations $R_{i}$:
$$F_i =  \{ (x_{0}, \ldots, x_{m_{i}-1}, (\lambda~\iota_{i}~(x_{0},\ldots, x_{m_{i}-1}), \lambda, \lambda)) \mid  x_{0}, \ldots, x_{m_{i}-1} \in \Sigma^{\star} \}.$$
Note that for $\bar{x} \in \Sigma^{\star}$,  $R_{i}(\bar{x})$ if and only if 
$F_{i}  (\bar{x}, (\lambda ~\iota_{i}~\bar{x}, \lambda, \lambda))$ holds and all $E_{i}$ chains emanating from $(\lambda~\iota_{i}~\bar{x}, \lambda, \lambda)$ are finite. 
We have built  the automatic structure $$\A= (\Sigma^{\star} \cup \cup_i D_{i},  E_{1}, \ldots, E_{n}, F_{1}, \ldots, F_{n}).$$ Two technical lemmas are used to show that the Scott rank of $\A$ is close to $\alpha$:

\begin{lemma}\label{lm:EquivTransfer}
For $\bar{x}, \bar{y}$ in the domain of $\C$ and for ordinal $\alpha$, if $\bar{x} \equiv_{\C}^{\alpha} \bar{y}$ then $\bar{x} \equiv_{\A}^{\alpha} \bar{y}$.
\end{lemma}

\begin{proof}
Let $X = \dom{\A} \setminus \Sigma^{\star}$.  We prove the stronger result that for any ordinal $\alpha$, and for all $\bar{x}, \bar{y} \in \Sigma^{\star}$ and $\bar{x}', \bar{y}' \in X$, if the following assumptions hold
\begin{enumerate}
\item \label{asmp:C}$\bar{x} \equiv_{\C}^{\alpha} \bar{y}$;
\item \label{asmp:I}$\la \bar{x}', E_{i} : i =1 \ldots n\ra_{\A} \cong_{f}   \la \bar{y}', E_{i}, : i =1 \ldots n\ra_{\A}$ (hence the substructures in $A$ are isomorphic) with $f(\bar{x}') = \bar{y}'$; and
\item \label{asmp:E}for each $x'_{k} \in \bar{x}'$, each $i=1, \ldots, n$ and each subsequence of indices of length $m_{i}$, 
\[
x'_{k} = (\lambda ~\iota_{i}~\bar{x}_{j}, \lambda, \lambda) ~~ \iff ~~ y'_{k} = (\lambda ~\iota_{i}~\bar{y}_{j}, \lambda, \lambda)
\]
\end{enumerate} 
then $\bar{x} \bar{x}' \equiv_{\A}^{\alpha} \bar{y} \bar{y}'$.  The lemma follows if we take $\bar{x}' = \bar{y}' = \lambda$ (the empty string).

We show the stronger result by induction on $\alpha$.  If $\alpha = 0$, we need to show that for each $i,k, k', k_{0}, \ldots, k_{m_{i}-1}$, 
\[
E_{i}(x'_{k}, x'_{k'}) \iff E_{i}(y'_{k}, y'_{k'}), 
\]
and that 
\[
F_{i}(x_{k_{0}}, \ldots, x_{k_{m_{i}-1}}, x'_{k'}) \iff F_{i}(y_{k_{0}}, \ldots, y_{k_{m_{i}-1}}, y'_{k'}).  
\]
The first statement follows by assumption \ref{asmp:I}, since the isomorphism must preserve the $E_{i}$ relations and maps $\bar{x}'$ to $\bar{y}'$.  The second statement follows by assumption \ref{asmp:E}.  

Assume now that $\alpha >0$ and that the result holds for all $\beta < \alpha$.  Let $\bar{x}, \bar{y} \in \Sigma^{\star}$ and $\bar{x}', \bar{y}' \in A$ be such that the assumptions of the lemma hold.  We will show that $\bar{x} \bar{x}' \equiv_{\A}^{\alpha} \bar{y} \bar{y}'$.  Let $\beta < \alpha$ and suppose $\bar{u} \in \Sigma^{\star}, \bar{u}' \in A$.  By assumption \ref{asmp:C}, there is $\bar{v} \in \Sigma^{\star}$ such that $\bar{x}\bar{u} \equiv_{\C}^{\beta} \bar{y} \bar{v}$.  By the construction (in particular, the smoothing steps), we can find a corresponding $\bar{v}' \in A$ such that assumptions \ref{asmp:I}, \ref{asmp:E} hold.  Applying the inductive hypothesis, we get that $\bar{x} \bar{u} \bar{x}'\bar{u}' \equiv_{\A}^{\beta} \bar{y} \bar{v}\bar{y}' \bar{v}'$.  Analogously, given $\bar{v}, \bar{v}'$ we can find the necessary $\bar{u}, \bar{u}'$.  Therefore, $\bar{x}\bar{x}' \equiv_{\A}^{\alpha}\bar{y} \bar{y}'$.
\end{proof}

\begin{lemma}\label{lm:CRankHigher}
If $\bar{x} \in \Sigma^{\star} \cup \cup_i D_i$, there is $\bar{y} \in \Sigma^{\star}$ with $\SR_{\A} (\bar{x} \bar{x}' \bar{u} ) \leq 2 + \SR_{\C} (\bar{y})$.
\end{lemma}
\begin{proof}
We use the notation $X_{P}$ to mean the subset of $X = A \setminus \Sigma^{\star}$ which corresponds to elements on fans associated with productive chains of the configuration space.  We write $X_{U}$ to mean the subset of $X$ which corresponds to the unproductive chains of the configuration space.  Therefore, $A = \Sigma^{\star} \cup X_{P} \cup X_{U}$, a disjoint union.  Thus, we will show that for each $\bar{x} \in \Sigma^{\star}$, $\bar{x}' \in X_{P}$, $\bar{u} \in X_{U}$ there is $\bar{y} \in \Sigma^{\star}$ such that $\SR_{\A} (\bar{x} \bar{x}' \bar{u} ) \leq 2 + \SR_{\C} (\bar{y})$.

Given $\bar{x}, \bar{x}', \bar{u}$, let $\bar{y} \in \Sigma^{\star}$ be a minimal element satisfying that $\bar{x} \subset \bar{y}$ and that $\bar{x}' \subset \la \bar{y}, E_{i}, F_{i} : i =1 \ldots n\ra_{\A} $.  Then we will show that $\bar{y}$ is the desired witness.  First, we observe that since the unproductive part of the structure is disconnected from the productive elements we can consider the two independently.  Moreover, because the structure of the unproductive part is predetermined and simple, for $\bar{u}, \bar{v} \in X_{U}$, if $\bar{u} \equiv_{\A}^{1} \bar{v}$ then $(\A, \bar{u}) \cong (\A, \bar{v})$. It remains to consider the productive part of the structure.

Consider any $\bar{z} \in \Sigma^{\star}$, $\bar{z}' \in X_{P}$ satisfying $\bar{z}' \subset \la \bar{z}, E_{i}, F_{i} : i =1 \ldots n\ra_{\A} $.  We claim that $SR_{\A}(\bar{z} \bar{z}') \leq 2 + \SR_{\C}(\bar{z})$.  It suffices to show that for all $\alpha$, for all $\bar{w} \in \Sigma^{\star}, \bar{w}' \in X_{P}$, 
\[
\bar{z}\bar{z}'~\equiv_{\A}^{2+\alpha}~\bar{w}\bar{w}' \qquad \implies \qquad \bar{z}~\equiv_{\C}^{\alpha}~\bar{w}.
\]
This is sufficient for the following reason. If $\bar{z} \bar{z}' \equiv_{\A}^{2 + \SR_{\C}(\bar{z})} \bar{w} \bar{w}'$ then $\bar{z} \equiv_{\C}^{\SR_{\C}(\bar{z})} \bar{w}$ and hence $(\C, \bar{z}) \cong (\C, \bar{w})$.  From this automorphism, we can define an automorphism of $\A$ mapping $\bar{z} \bar{z}'$ to $\bar{w} \bar{w}'$ because $\bar{z} \bar{z}' \equiv_{\A}^{2} \bar{w} \bar{w}'$  and hence for each $i$, the relative positions of $\bar{z}'$ and $\bar{w}'$ in the fans above $\bar{z}$ and $\bar{w}$ are isomorphic.  Therefore, $2 + \SR_{\C}(\bar{z}) \geq \SR_{\A}(\bar{z}\bar{z}')$.

So, we now show that for all $\alpha$, for all $\bar{w} \in \Sigma^{\star}, \bar{w}' \in X_{P}$, $\bar{z}\bar{z}' \equiv_{\A}^{2+\alpha} \bar{w} \bar{w}'$ implies that $\bar{z} \equiv_{\C}^{\alpha} \bar{w}$.   We proceed by induction on $\alpha$.  For $\alpha = 0$, suppose that $\bar{z}\bar{z}' \equiv_{\A}^{2} \bar{w} \bar{w}'$.  This implies that for each $i$ and for each subsequence of length $m_{i}$ of the indices, the $E_{i}$-fan above $\bar{z}_{j}$ has an infinite chain if and only if the $E_{i}$-fan above $\bar{w}_{j}$ does. Therefore, $R_{i}(\bar{z}_{j})$ if and only if $R_{i}(\bar{w}_{j})$.  Hence, $\bar{z} \equiv_{\C}^{0} \bar{w}$, as required. For the inductive step, we assume the result holds for all $\beta < \alpha$.  Suppose that $\bar{z}\bar{z}' \equiv_{\A}^{2+\alpha} \bar{w} \bar{w}'$.  Let $\beta < \alpha$ and $\bar{c} \in \Sigma^{\star}$.  Then $2 + \beta < 2 +\alpha$ so by definition there is $\bar{d} \in \Sigma^{\star}, \bar{d}' \in X_{P}$ such that $\bar{z}\bar{z}' \bar{c}\equiv_{\A}^{2+\beta} \bar{w} \bar{w}' \bar{d} \bar{d}'$.  However, since $2 + \beta > 1$, $\bar{d}'$ must be empty (elements in $\Sigma^{\star}$ cannot be $1$-equivalent to elements in $X_{P}$).   Then by the induction hypothesis, $\bar{z} \bar{c} \equiv_{\C}^{\beta} \bar{w} \bar{d}$.  The argument works symmetrically if we are given $\bar{d}$ and want to find $\bar{c}$.   Thus, $\bar{z} \equiv_{\C}^{\alpha} \bar{w}$, as required.
\end{proof}

Putting Lemmas \ref{lm:EquivTransfer} and \ref{lm:CRankHigher} together, we can prove the main result about our construction.

\begin{theorem}\label{thm:SR}  Let $\C$ be a computable structure and construct the automatic structure $\A$ from it as above.  
Then $\SR(\C) \leq \SR(\A) \leq 2 + \SR(\C)$.
\end{theorem}

\begin{proof}
Let $\bar{x}$ be a tuple in the domain of $\C$.  Then, by the definition of Scott rank, $\SR_{A}(\bar{x})$ is the least ordinal $\alpha$ such that for all $\bar{y} \in \dom(\A)$, $\bar{x} \equiv_{A}^{\alpha} \bar{y}$ implies that $(\A, \bar{x}) \cong (\A, \bar{y})$; and similarly for $\SR_{C}(\bar{x})$.  We first show that $\SR_{\A}(\bar{x}) \geq \SR_{\C}(\bar{x})$.  Suppose $\SR_{C}(\bar{x}) = \beta$.  We assume for a contradiction that $\SR_{A}(\bar{x})= \gamma < \beta$.  Consider an arbitrary $\bar{z} \in \Sigma^{\star}$ (the domain of $\C$) such that $\bar{x} \equiv_{\C}^{\gamma} \bar{z}$.  By Lemma \ref{lm:EquivTransfer}, $\bar{x} \equiv_{\A}^{\gamma} \bar{z}$.  But, the definition of $\gamma$ as the Scott rank of $\bar{x}$ in $\A$ implies that $(\A, \bar{x}) \cong (\A, \bar{z})$.  Now, $\C$ is $L_{\omega_{1}, \omega}$ definable in $\A$ and therefore inherits the isomorphism.  Hence, $(\C, \bar{x}) \cong (\C, \bar{z})$.  But, this implies that $\SR_{\C}(\bar{x}) \leq \gamma < \beta = \SR_{\C}(\bar{x})$, a contradiction.

So far, we have that for each $\bar{x} \in \Sigma^{\star}$, $\SR_{\A}(\bar{x}) \geq \SR_{\C}(\bar{x})$.  Hence, since $\dom(\C) \subset \dom(\A)$, 
\begin{align*}
\SR (\A) &= \sup\{\SR_{\A}(\bar{x}) +1: \bar{x} \in \dom(\A)\} \\
&\geq \sup\{\SR_{\A}(\bar{x}) +1: \bar{x} \in \dom(\C)\} \\
&\geq \sup\{\SR_{\C}(\bar{x}) +1: \bar{x} \in \dom(\C)\} = \SR(C).
\end{align*}

In the other direction, we wish to show that $ \SR(\A) \leq 2+ \SR(\C)$.  Suppose this is not the case.  Then there is $\bar{x} \bar{x}' \bar{u} \in \A$ such that $\SR_{\A}(\bar{x} \bar{x}' \bar{u} ) \geq 2 + \SR(\C)$.  By Lemma \ref{lm:CRankHigher}, there is $\bar{y} \in \Sigma^{\star}$ such that $2 + \SR_{\C} (\bar{y}) \geq 2 + \SR(\C)$, a contradiction.
\end{proof}

Recent work in the theory of computable structures has focussed on finding computable structures of high Scott rank.  Nadel \cite{Nad85} proved that any computable structure has Scott rank at most $\omega_{1}^{CK} + 1$. Early on, Harrison \cite{Harr68} showed that there is a computable ordering of type $\omega_{1}^{CK}( 1 + \eta)$ (where $\eta$ is the order type of the rational numbers). This ordering has Scott rank $\omega_{1}^{CK}+1$, as witnessed by any element outside the initial $\omega_{1}^{CK}$ set.   However, it was not until much more recently that a computable structure of Scott rank $\omega_{1}^{CK}$ was produced (see Knight and Millar \cite{KnM}).  A recent result of Cholak, Downey, and Harrington gives the first natural example of a structure with Scott rank $\omega_1^{CK}$: the computable enumerable sets under inclusion \cite{CDH}.

\begin{corollary}
There is an automatic structure with Scott rank $\omega_1^{CK}$.  There is an automatic structure with Scott rank $\omega_1^{CK}+1$.
\end{corollary}

We also apply the construction to \cite{GonKn02}, where it is proved that there are computable structures with Scott ranks above each computable ordinal.  In this case, we get the following theorem.\\

\noindent{\bf Theorem \ref{thm:ScottRank}.~}{\em For each computable ordinal $\alpha$, there is an automatic structure of Scott rank at least $\alpha$.}

\section{Cantor-Bendixson Rank of Automatic Successor Trees}\label{s:CBdefs}

In this section we show that there are automatic successor trees of high Cantor-Bendixson (CB) 
rank. Recall the definitions of partial order trees and successor trees from Section \ref{s:Intro}.
Note that if $(T,\leq)$ is an automatic partial order tree then the successor tree $(T,S)$, where the relation $S$ is defined by $S(x,y) \iff (x <  y) \  \& \ \neg \exists z (x < z < y)$, is automatic. 

\begin{definition}
The {\bf derivative} of a (partial order or successor) tree $T$, $d(T)$, is the subtree of $T$ whose domain is 
\[
\{ x \in T: \text{ $x$ lies on at least two infinite paths in $T$}\}.
\]
By induction, $d^{0}(T) = T$, $d^{\alpha+1} (T) = d(d^{\alpha}(T))$, and for $\gamma$ a limit ordinal, $d^{\gamma}(T) = \cap_{\beta < \gamma} d^{\beta}(T)$.  The {\bf CB rank} of the tree, $CB(T)$, is the least $\alpha$ such that $d^{\alpha}(T) = d^{\alpha+1}(T)$.
\end{definition}

The CB ranks of automatic partial order trees are finite \cite{KhRS03}.  This is not true of automatic successor trees.
The main theorem of this section provides a general technique for building trees of given CB ranks. Before we get to it, we give some examples of automatic successor ordinals whose CB ranks are low.

\begin{example} 
There is an automatic partial order tree (hence an automatic successor tree) whose CB rank is $n$ for each $n \in \omega$.
\end{example}

\begin{proof}
The tree $T_n$ is defined over the $n$ letter alphabet $\{a_1,\ldots, a_n\}$ as follows. The domain of the tree is  $a_{1}^{\star} \cdots a_{n}^{\star}$. The order $\leq_n$ is the prefix partial order. Therefore, the successor relation is given as follows:
\begin{equation*}
S(a_{1}^{\ell_{1}} \cdots a_{i}^{\ell_{i}}) = \begin{cases}
	\{ a_{1}^{\ell_{1}} \cdots a_{i}^{\ell_{i}+1}, a_{1}^{\ell_{1}} \cdots a_{i}^{\ell_{i}}a_{i+1} \} \qquad \text{if $1 \leq i <n$} \\
	\{ a_{1}^{\ell_{0}} \cdots a_{i}^{\ell_{i}+1} \} \qquad \qquad \text{if $i=n$}
\end{cases}
\end{equation*}

Note that if $n=0$ then the tree is empty, which is consistent with it having CB rank $0$. 
It is obvious that $T_n$ is an automatic partial order tree. The rank of $T_{n}$ can be shown, by induction, to be equal to $n$.
 \end{proof}

The following examples code the finite rank successor trees uniformly into one automaton in order to push the rank higher.  We begin by building an automatic successor tree $T_{\omega+1}$ of rank $\omega+1$. We note that the CB ranks of all trees with at most countably many paths are successor ordinals. Thus, 
$T_{\omega+1}$ will have countably many paths. Later, we construct a tree of rank $\omega$ which must embed the perfect
tree because its CB rank is a limit ordinal.

\begin{example} 
There is an automatic successor tree $T_{\omega+1}$ whose CB rank is $\omega+1$.
\end{example}

\begin{proof}
Informally, this tree is a chain of trees of increasing finite CB ranks.  Let $T_{\omega+1} = (\{0,1\}^{\star}, S) $ with $S$ defined as follows:
\begin{align*}
S(1^{n}) = \{ 1^{n}0, 1^{n+1} \} \qquad &\text{for all $n$}\\ 
S(0u) = 0u0 \qquad &\text{for all $u \in \{0,1\}^{\star}$}\\
S(1^{n}0 u) = \{ 1^{n}0u0, 1^{n-1} 0 u 1 \} \qquad &\text{for $n \geq 1$ and $u \in \{0,1\}^{\star}$}
\end{align*}

Intuitively, the subtree of rank $n$ is coded by the set $X_n$ of nodes which contain exactly $n$  $1$s.
By induction on the length of strings, we can show that $\text{range}(S) = \{0,1\}^{\star}$ and hence the domain of the tree is also $\{0,1\}^{\star}$.  It is also not hard to show that the transitive closure of the relation $S$ satisfies the conditions of being a tree. It is also easy to see that $T_{\omega+1}$ is automatic.
Finally, we compute the rank of $T_{\omega + 1}$.  We note that in successive derivatives, each of the finite rank sub-trees $X_n$ is reduced in rank by $1$.  Therefore 
\[
d^{\omega}(T) = 1^{\star}.
\]
But, since each point in $1^{\star}$ is on exactly one infinite path, $d^{\omega+1}(T) = \emptyset$, and this is a fixed-point.  Thus, $CB(T_{\omega+1}) = \omega+1$, as required.
\end{proof}

The following example gives a tree $T_{\omega}$ of rank $\omega$. The idea is to code the trees $T_n$ provided above into the leftmost path of the full binary tree.

\begin{example}\label{ex:LimitOrd}
There is an automatic successor tree $T_{\omega}$ whose CB rank is $\omega$.
\end{example}

\begin{proof}
The tree is the full binary tree, where at each node on the leftmost branch we append trees of increasing finite CB rank. Thus, 
define $T_{\omega} = ( \{0,1\}^{\star} \cup \{0,a\}^{\star}, S)$ where $S$ is given as follows:
\begin{align*}
S(u1v) = \{ u1v0, u1v1 \} \qquad &\text{for all $u,v \in \{0,1\}^{\star}$}\\ 
S(0^{n}) = \{0^{n+1}, 0^{n} 1, 0^{n} a\} \qquad &\text{for all $n$} \\
S(au) = aua \qquad &\text{for all $u \in \{0,a\}^{\star}$}\\
S(0^{n}a u) = \{ 0^{n}aua, 0^{n-1} a u 0 \} \qquad &\text{for $n \geq 1$ and $u \in \{0,a\}^{\star}$}
\end{align*}

Proving that $T_{\omega}$ is an automatic successor tree is a routine check.   So, we need only compute its rank.  
Each derivative leaves the right part of the tree (the full binary tree) fixed.  However, the trees appended to the leftmost path of the tree are affected by taking derivatives. Successive derivatives decrease the rank of the protruding finite rank trees by $1$.  Therefore, 
$d^{\omega}(T_{\omega}) = \{0,1\}^{\star}$, a fixed point.  Thus, 
$CB(T_{\omega}) = \omega$.
\end{proof}

To extend these examples to higher ordinals, we consider the {\bf product} operation on 
trees defined as follows. Let $(T_1, S_1)$ and $(T_2,S_2)$ be successor trees. The product of these trees is the tree $(T,S)$ with domain $T=T_1\times T_2$ and successor relations given by:
\begin{align*}
S((x,y), (u,v))\iff\begin{cases}
y ~\text{is root of $T_{2}$~ and $\left( u=x, S_{2}(y,v)\right)$ or $\left( S_{1}(x,u),  y=v\right)$} \\
y ~\text{is not the root of $T_{2}$ and $u=x$, $S_{2}(y,v)$}.
\end{cases}
\end{align*}

The following is an easy proposition.
\begin{proposition}
Assume that $T_1$ and $T_2$ are successor trees of CB ranks $\alpha$ and $\beta$, respectively, each having at most countably many paths. Then $T_1\times T_2$ has 
CB rank $\alpha +\beta$. Moreover, if $T_1$ and $T_2$ are automatic successor trees then so is the product. 
\end{proposition}

The examples and the proposition above yield tools for building automatic successor trees of CB ranks up to $\omega^2$. However, it is not clear that these methods can be applied to obtain automatic successor trees of higher CB ranks.  We will see that a different approach to building automatic successor trees will yield all possible CB ranks.

We are now ready for the main theorem of this section. As before, we will transfer results from computable trees to automatic trees. We note that every computable successor 
tree $(T,S)$ is also a computable partial order tree. Indeed, in order to compute
if $x\prec_S y$, we effectively find the distances of $y$ and $x$ from the root. If $y$ is closer to the root or is at the same distance as $x$ then $\neg (x\prec_S y)$; otherwise,
we start computing the trees above all $z$ at the same distance from the root as $x$ is.
Then $y$ must appear in one of these trees. Thus, we have computed whether $x\prec_S y$. Hence, every computable successor tree is a computable partial order tree.  However, not every computable partial order tree is a computable successor tree.  We have the following inclusions:
\[
\text{Aut.\ PO trees} \subset \text{Aut.\ Succ.\ trees} \subset \text{Comp.\ Succ.\ trees} \subset \text{Comp.\ PO trees}
\]
We use the fact that for each $\alpha < \omega_1^{CK}$  there is a computable successor tree of CB rank $\alpha$.  This fact can be proven by recursively coding up computable trees of increasing CB rank.

\vspace{5pt}

\noindent{\bf Theorem \ref{thm:RecTrees}.~}{\em
For $\alpha < \omega_1^{CK}$ there is an automatic successor tree of CB rank $\alpha$.}

\begin{proof} 
Suppose we are given $\alpha < \omega_{1}^{CK}$.  Take a computable tree $R_{\alpha}$ of CB rank $\alpha$.  We use the same construction as in the case of well-founded relations (see the proof of Theorem \ref{thm:HeightRank}).
The result is
a stretched out version of the tree $R_{\alpha}$, where between each two elements of the original tree we have a coding of their computation.  In addition, extending from each $x \in \Sigma^{\star}$ we have infinitely many finite computation chains.  Those chains which correspond to output \lqt no\rqt are not connected to any other part of the automatic structure.  Finally, there is a disjoint part of the structure consisting of chains whose bases are not valid initial configurations.  By the reversibility assumption, each unproductive component of the configuration space is isomorphic either to a finite chain or to an $\omega$-chain.  Moreover,  the set of invalid initial configurations which are the base of such an unproductive chain is regular.  We connect all such bases of unproductive chains to the root and get an automatic successor tree, $T_{\alpha}$.

We now consider the CB rank of $T_{\alpha}$.  Note that the first derivative removes 
all the subtrees whose roots are at distance $1$ from the root  and are invalid initial computations.  This occurs because each of the invalid computation chains has no branching and is not connected to any other element of the tree.  Next, if we consider the subtree of $T_{\alpha}$ rooted at some $x \in \Sigma^{\star}$, we see that all the paths which correspond to computations whose output is \lqt no\rqt vanish after the first derivative.  Moreover, $x \in d(T_{\alpha})$ if and only if $x \in d(R_{\alpha})$ because the construction did not add any new infinite paths.  Therefore, after one derivative, the structure is exactly a stretched out version of $d(R_{\alpha})$.  Likewise, for all $\beta < \alpha$, $d^{\beta}(T_{\alpha})$ is a stretched out version of $d^{\beta}(R_{\alpha})$. Hence, $CB(T_{\alpha}) = CB(R_{\alpha}) = \alpha$.
\end{proof}

Automatic successor trees have also been recently studied by Kuske and Lohrey in \cite{KuLoh}.  In that paper, techniques similar to those above are used to show that the existence of an infinite path in an automatic successor tree is $\Sigma_1^1$-complete.  In addition, Kuske and Lohrey look at graph questions for automatic graphs and show that the existence of a Hamiltonian path is $\Sigma_1^1$-complete whereas the set cover problem is decidable.

\section{Conclusion}
This paper studies the complexity of automatic structures.  In particular, we seek to understand the difference in complexity between automatic and computable structures.  We show that automatic well-founded partial orders are considerably simpler than their computable counterparts, because the ordinal heights of automatic partial orders are bounded below $\omega^\omega$.  On the other hand, computable well-founded relations, computable successor trees, and computable structures in general can be transformed into automatic objects in a way which (almost) preserves the ordinal height, Cantor-Bendixson rank, or Scott ranks (respectively).  Therefore, the corresponding classes of automatic structures are as complicated as possible.

\end{document}